\documentclass[12pt,reqno]{amsart}
\usepackage{amssymb}
\usepackage{amsxtra}
\usepackage{mathrsfs}
\usepackage[all]{xy}
\usepackage{cite}
\usepackage{paralist}
\usepackage[unicode,hypertex]{hyperref}
\newtheorem{theorem}{Theorem}[section]
\newtheorem{lemma}[theorem]{Lemma}
\newtheorem{prop}[theorem]{Proposition}
\newtheorem{corollary}[theorem]{Corollary}
\newtheorem{problem}[theorem]{Problem}
\theoremstyle{definition}
\newtheorem{definition}[theorem]{Definition}
\theoremstyle{remark}
\newtheorem{example}[theorem]{Example}
\newtheorem{remark}[theorem]{Remark}
\newtheorem*{ackn}{Acknowledgements}

\DeclareMathOperator{\Tor}{Tor}
\DeclareMathOperator{\Ext}{Ext}
\DeclareMathOperator{\VdB}{VdB}
\DeclareMathOperator{\codim}{codim}
\newcommand*{\ptens}[1]{\mathop{\widehat\otimes}_{#1}}

\newcommand*{\Ptens}{\mathop{\widehat\otimes}}
\newcommand*{\tens}[1]{\mathop{\otimes}_{#1}}

\newcommand*{\lmod}{\mbox{-}\!\mathop{\mathsf{mod}}}
\newcommand*{\rmod}{\mathop{\mathsf{mod}}\!\mbox{-}}
\newcommand*{\bimod}{\mbox{-}\!\mathop{\mathsf{mod}}\!\mbox{-}}
\newcommand*{\Vect}{\mathsf{Vect}}
\newcommand*{\id}{1}

\renewcommand*{\dh}{\mathop{\mathrm{dh}}}
\newcommand*{\db}{\mathop{\mathrm{db}}}
\newcommand*{\dg}{\mathop{\mathrm{dg}}}
\newcommand*{\wdh}{\mathop{\mathrm{w.dh}}}
\newcommand*{\wdb}{\mathop{\mathrm{w.db}}}
\newcommand*{\wdg}{\mathop{\mathrm{w.dg}}}

\newcommand*{\h}{\mathbf h}

\newcommand*{\CC}{\mathbb C}
\newcommand*{\N}{\mathbb N}
\newcommand*{\Z}{\mathbb Z}

\newcommand*{\cO}{\mathscr O}
\newcommand*{\cP}{\mathscr P}
\newcommand*{\cF}{\mathscr F}
\newcommand*{\cG}{\mathscr G}
\newcommand*{\cI}{\mathscr I}
\newcommand*{\cH}{\mathscr H}
\newcommand*{\cN}{\mathscr N}
\newcommand*{\cL}{\mathscr L}
\newcommand*{\cT}{\mathscr T}
\newcommand*{\cLB}{\mathscr{LB}}
\DeclareMathOperator{\Ker}{Ker}
\DeclareMathOperator{\Coker}{Coker}

\renewcommand*{\Im}{\mathop{\mathrm{Im}}}
\newcommand*{\cIm}{\mathop{\mathscr Im}}
\newcommand*{\cTor}{\mathop{\mathscr T\!or}\nolimits}
\newcommand*{\cHom}{\mathop{\mathscr H\!om}\nolimits}
\newcommand*{\cExt}{\mathop{\mathscr Ext}\nolimits}

\newcommand*{\eps}{\varepsilon}
\newcommand*{\ol}{\overline}
\newenvironment{mycompactenum}{\pltopsep=5pt\begin{compactenum}[\upshape (i)]}%
{\end{compactenum}}
\newcommand*{\lar}{\leftarrow}

\newcommand*{\xla}{\xleftarrow}
\newcommand*{\xra}{\xrightarrow}
\begin{document}
\title[Homological dimensions of modules of holomorphic functions]{Homological
dimensions\\ of modules of holomorphic functions\\ on submanifolds of Stein manifolds}
\subjclass[2010]{Primary 46M18; secondary 46H25, 46E25, 16E10, 16E30, 16E40}
\author{A. Yu. Pirkovskii}
\address{Faculty of Mathematics\\
National Research University Higher School of Economics\\
Vavilova 7, 117312 Moscow, Russia}
\email{aupirkovskii@hse.ru, pirkosha@online.ru}
\thanks{This work was partially supported by the Ministry of Education
and Science of Russia (programme ``Development of the scientific
potential of the Higher School'', grant no. 2.1.1/2775).}
\date{}
\begin{abstract}
Let $X$ be a Stein manifold, and let $Y\subset X$ be a closed complex
submanifold. Denote by $\cO(X)$ the algebra of holomorphic functions on $X$.
We show that the weak (i.e., flat) homological dimension of
$\cO(Y)$ as a Fr\'echet $\cO(X)$-module equals the codimension of $Y$ in $X$.
In the case where $X$ and $Y$ are of Liouville type, the same formula
is proved for the projective homological dimension of $\cO(Y)$ over $\cO(X)$.
On the other hand, we show that if $X$ is of Liouville type and
$Y$ is hyperconvex, then the projective homological dimension of $\cO(Y)$
over $\cO(X)$ equals the dimension of $X$.
\end{abstract}
\maketitle

\section{Introduction}

This paper is motivated by the following fact from commutative algebra.
Let $X$ be a nonsingular affine algebraic variety over $\CC$, and let
$\cO(X)$ denote the algebra of regular functions on $X$.
It is well known and easy to show
that, for each nonsingular closed algebraic subvariety $Y\subset X$, the projective
homological dimension of $\cO(Y)$ considered as a module over $\cO(X)$ is equal
to the codimension of $Y$ in $X$:
\begin{equation}
\label{motiv}
\dh\nolimits_{\cO(X)} \cO(Y)=\codim_X Y.
\end{equation}
Indeed, let $\cI_Y\subset\cO_X$ denote the ideal sheaf of $Y$ in $X$.
Since $Y$ is a local complete intersection in $X$ \cite[8.22.1]{Hart},
it follows that for each $y\in Y$ the ideal $\cI_{Y,y}$ of the local ring $\cO_{X,y}$
is generated by a regular sequence of length $m=\codim_X Y$.
Hence we have $\dh_{\cO_{X,y}}\cO_{Y,y}=m$
\cite[3.8, Theorem 22]{Northcott2}. The global formula~\eqref{motiv} now follows
from~\cite[9.2, Theorem 11]{Northcott}. Note also that, since $\cO(X)$ is
Noetherian, we also have
\begin{equation}
\label{motiv2}
\wdh\nolimits_{\cO(X)} \cO(Y)=\codim_X Y,
\end{equation}
where $\wdh$ stands for the weak (i.e., flat) homological dimension.

In \cite{Pir_Nova}, we proved \eqref{motiv} in the
situation where $X$ is a smooth real manifold, $\cO(X)=C^\infty(X)$ is
the algebra of smooth functions on $X$, and $Y\subset X$
is a closed smooth submanifold. Now the formula~\eqref{motiv}
should be understood in the context of ``Topological Homology'', i.e.,
a relative homological algebra in categories of Fr\'echet modules over
Fr\'echet algebras \cite{X1}.
The above-mentioned localization technique is not applicable in this case,
so we had to develop an essentially different proof which heavily relied
on the softness of the structure sheaf $C^\infty_X$.
The formula~\eqref{motiv2} also easily follows from the results
of~\cite{Pir_Nova} (see Remark~\ref{rem:wdh_C^infty} for details).

Our goal here is to study complex analytic analogues of~\eqref{motiv}
and~\eqref{motiv2} in the context of Topological Homology.
Suppose that $X$ is a complex Stein manifold,
$\cO(X)$ is the Fr\'echet algebra of holomorphic functions on $X$, and
$Y\subset X$ is a closed analytic submanifold.
In Section~\ref{sect:weak}, we show that~\eqref{motiv2} holds.
As a byproduct, we obtain a complex analytic version of the
Hochschild-Kostant-Rosenberg Theorem~\cite{HKR}. Our proof essentially
uses the nuclearity of $\cO(X)$ and some results of O.~Forster~\cite{For}.

Surprisingly, the validity of \eqref{motiv} turns out to depend on
some special properties of $X$ and $Y$. In Section~\ref{sect:Liouv},
we show that \eqref{motiv} holds
provided that both $X$ and $Y$ are of Liouville type
(i.e., each bounded above plurisubharmonic function on $X$ and $Y$ is constant).
In particular, this is true whenever both $X$ and $Y$ are affine algebraic.
On the other hand, we show in Section~\ref{sect:hyper} that~\eqref{motiv}
fails in the general case. Specifically, assume that $X$ is of Liouville type
and that $Y$ is hyperconvex (this means that there exists a negative plurisubharmonic
exhaustion function on $Y$). We show that
$\dh_{\cO(X)} \cO(Y)=\dim X$ in this case.
For example, if $Y$ is the open disc embedded into $\CC^2$ (cf. \cite{Alexander}),
then $\dh_{\cO(\CC^2)} \cO(Y)=2$.

Our proofs heavily rely on the linear topological invariants $(DN)$, $(\Omega)$,
and $(\ol{\Omega})$, introduced by D.~Vogt in the 1970ies
(see, e.g., \cite[Chapter 8]{MV}). We essentially use the splitting theorem due to
D.~Vogt and M.~J.~Wagner~\cite{VW_quot}, the results of D.~Vogt \cite{Vogt_bdd} on bounded
linear maps between Fr\'echet spaces, and the results of V.~P.~Zakharyuta~\cite{Zah_iso},
D.~Vogt~\cite{Vogt_some_results}, and A.~Aytuna~\cite{Aytuna_sp_anal}
on linear topological properties
of spaces of holomorphic functions. Another essential ingredient is the
Van den Bergh isomorphism~\cite{VdB} between the Hochschild homology and
cohomology of $\cO(X)$, which is proved in Section~\ref{sect:VdB}.

\section{Preliminaries}
\label{sect:prelim}

This section gives a brief account of some basic facts from Topological Homology.
Our main reference is \cite{X1}; some details can also be found in
\cite{X2,X_HOA,T1,Eschm_Put,Pir_msb}.

Throughout, all vector spaces and algebras are assumed to be over the field $\CC$
of complex numbers. All algebras are assumed to be associative and
unital.
By a {\em Fr\'echet algebra} we mean an algebra $A$ endowed with
a complete, metrizable locally convex topology
(i.e., $A$ is an algebra and a Fr\'echet space simultaneously)
such that the product map $A\times A\to A$ is continuous.
If, in addition, $A$ is locally $m$-convex
(i.e., the topology on $A$ can be determined by a family
$\{\|\cdot\|_\lambda : \lambda\in\Lambda\}$ of
seminorms satisfying $\| ab\|_\lambda\le\| a\|_\lambda \| b\|_\lambda$
for all $a,b\in A$), then $A$ is said to be a
{\em Fr\'echet-Arens-Michael algebra}.

Let $A$ be a Fr\'echet algebra.
A {\em left Fr\'echet $A$-module} is a left $A$-module $M$ endowed with
a complete, metrizable locally convex topology in such a way that
the action $A\times M\to M$ is continuous.
We always assume that $1_A\cdot x=x$ for all $x\in M$, where $1_A$ is the identity of $A$.
Left Fr\'echet $A$-modules and their continuous morphisms form a category
denoted by $A\lmod$. Given $M,N\in A\lmod$,
the space of morphisms from $M$ to $N$ will be denoted by $\h_A(M,N)$.
We always endow $\h_A(M,N)$ with the topology of uniform convergence on bounded
subsets of $M$; note that this topology, in general, is not metrizable.
The categories $\rmod A$ and $A\bimod A$ of right
Fr\'echet $A$-modules and of Fr\'echet $A$-bimodules are defined similarly.
Note that $A\bimod A\cong A^e\lmod\cong\rmod A^e$, where $A^e=A\Ptens A^{\mathrm{op}}$,
and where $A^{\mathrm{op}}$ stands for the algebra opposite to $A$.

If $M$ is a right Fr\'echet $A$-module and $N$
is a left Fr\'echet $A$-module, then their {\em $A$-module tensor product}
$M\ptens{A}N$ is defined to be
the quotient $(M\Ptens N)/L$, where $L\subset M\Ptens N$
is the closed linear span of all elements of the form
$x\cdot a\otimes y-x\otimes a\cdot y$
($x\in M$, $y\in N$, $a\in A$).
As in pure algebra, the $A$-module tensor product can be characterized
by the universal property that, for each Fr\'echet space $E$,
there is a natural bijection between the set of all
continuous $A$-balanced bilinear maps from $M\times N$ to $E$
and the set of all continuous linear maps from
$M\ptens{A}N$ to $E$.

A chain complex $C=(C_n,d_n)_{n\in\Z}$ in $A\lmod$ is {\em admissible} if
it splits in the category of topological vector spaces, i.e., if it has
a contracting homotopy consisting of continuous linear maps. Geometrically,
this means that $C$ is exact, and $\Ker d_n$ is a complemented subspace
of $C_n$ for each $n$.

Let $\Vect$ denote the category of vector spaces and linear maps.
A left Fr\'echet $A$-module $P$ is {\em projective}
(respectively, {\em strictly projective})
if the functor
$\h_A(P,-)\colon A\lmod\to\Vect$ takes admissible (respectively, exact) sequences
of Fr\'echet $A$-modules to exact sequences of vector spaces.
Similarly, a left Fr\'echet $A$-module $F$ is {\em flat}
(respectively, {\em strictly flat}) if the tensor product functor
$(-)\ptens{A} F\colon \rmod A\to\Vect$
takes admissible (respectively, exact) sequences
of Fr\'echet $A$-modules to exact sequences of vector spaces.
Clearly, each strictly projective (respectively, strictly flat)
Fr\'echet module is projective (respectively, flat).
It is also known that every projective Fr\'echet module is flat.

A {\em resolution} of $M\in A\lmod$ is a pair $(P,\eps)$
consisting of a nonnegative chain complex
$P$ in $A\lmod$ and a morphism $\eps\colon P_0\to M$ making the sequence
$P\xra{\eps} M\to 0$ into an admissible complex.
The {\em length} of $P$ is the minimum integer $n$
such that $P_i=0$ for all $i>n$, or $\infty$ if there is no such $n$.
If all the $P_i$'s are projective (respectively, flat), then
$(P,\eps)$ is called a {\em projective resolution}
(respectively, a {\em flat resolution}) of $M$.
It is a standard fact that $A\lmod$ has {\em enough projectives},
i.e., each left Fr\'echet $A$-module has a projective resolution.
The same is true of $\rmod A$ and $A\bimod A$.

If $M,N\in A\lmod$, then the space $\Ext^n_A(M,N)$ is defined to be the $n$th
cohomology of the complex $\h_A(P,N)$, where $P$ is a projective
resolution of $M$. Similarly, if $M\in\rmod A$ and $N\in A\lmod$, then
the space $\Tor_n^A(M,N)$ is defined to be the $n$th
homology of the complex $M\ptens{A} F$, where $F$ is a flat
resolution of $N$. The spaces $\Ext^n_A(M,N)$ and $\Tor_n^A(M,N)$
do not depend on the particular choice of $P$ and $F$
and have the usual functorial properties (see \cite{X1} for details).
If $M\in A\bimod A$, then the {\em $n$th
Hochschild cohomology} (respectively, \emph{homology}) of $A$ with coefficients in $M$
is defined by $\cH^n(A,M)=\Ext^n_{A^e}(A,M)$
(respectively, $\cH_n(A,M)=\Tor_n^{A^e}(M,A)$).

For each $M\in\rmod A$ and each $N\in A\lmod$ the tensor product $N\Ptens M$
is a Fr\'echet $A$-bimodule in a natural way, and there exist topological
isomorphisms
\begin{equation}
\label{Tor_Hoch}
\Tor_n^A(M,N)\cong\cH_n(A,N\Ptens M).
\end{equation}
If $M,N\in A\lmod$ and $M$ is a Banach module, then $\cL(M,N)$ is a Fr\'echet
$A$-bimodule in a natural way, and we have vector space isomorphisms
\begin{equation}
\label{Ext_Hoch}
\Ext^n_A(M,N)\cong\cH^n(A,\cL(M,N)).
\end{equation}

In some cases, the spaces $\Tor_n^A(M,N)$ can be computed via nonadmissible
resolutions. Let $N\in A\lmod$, and let $(F,\eps)$ be a pair
consisting of a nonnegative chain complex
$F$ in $A\lmod$ and a morphism $\eps\colon F_0\to N$ making the sequence
$F\xra{\eps} N\to 0$ into an exact complex.
Suppose also that all the $F_i$'s are flat Fr\'echet modules (so that $F$
is ``almost'' a flat resolution of $N$).
Take $M\in\rmod A$, and assume that either all the $F_i$'s are nuclear,
or both $A$ and $M$ are nuclear. Then we have
$\Tor_n^A(M,N)\cong H_n(M\ptens{A}F)$ (see \cite{T1} or \cite[3.1.13]{Eschm_Put}).

The {\em projective homological dimension} of $M\in A\lmod$ is the
minimum integer $n=\dh_A M\in\Z_+\cup\{\infty\}$
with the property that $M$ has a projective resolution of length $n$.
Similarly, the {\em weak homological dimension} of $M\in A\lmod$ is the
minimum integer $n=\wdh_A M\in\Z_+\cup\{\infty\}$
with the property that $M$ has a flat resolution of length $n$.
Equivalently,
\[
\begin{split}
\dh\nolimits_A M&=\min\{ n\in\Z_+ | \Ext_A^{n+1}(M,N)=0\;\forall\, N\in A\lmod\}\\
&=\min\{ n\in\Z_+ | \Ext_A^{p+1}(M,N)=0\;\forall\, N\in A\lmod,\;\forall p\ge n\};\\
\wdh\nolimits_A M &= \min\left\{ n\in\Z_+ \left|\;
\parbox{65mm}{%
$\Tor_{n+1}^A(N,M)=0$, and $\Tor_n^A(N,M)$\\
is Hausdorff $\forall\, N\in\rmod A$}
\right.\right\}\\
&= \min\left\{ n\in\Z_+ \left|\;
\parbox{65mm}{%
$\Tor_{p+1}^A(N,M)=0$, and $\Tor_n^A(N,M)$\\
is Hausdorff $\forall\, N\in\rmod A,\;\forall p\ge n$}
\right.\right\}.
\end{split}
\]
Note that $\dh_A M=0$ if and only if $M$ is projective, and
$\wdh_A M=0$ if and only if $M$ is flat.
Since each projective module is flat, we clearly have $\wdh_A M\le\dh_A M$.

The {\em global dimension} and the {\em weak global dimension} of $A$
are defined by
\begin{align*}
\dg A&=\sup\{ \dh\nolimits_A M \,|\, M\in A\lmod\},\\
\wdg A&=\sup\{ \wdh\nolimits_A M \,|\, M\in A\lmod\}.
\end{align*}
The {\em bidimension} and the {\em weak bidimension} of
$A$ are defined by
$\db A=\dh_{A^e} A$ and $\wdb A=\wdh_{A^e} A$, respectively.
We clearly have $\wdg A\le\dg A$ and $\wdb A\le\db A$.
It is also true (but less obvious) that $\dg A\le\db A$ and $\wdg A\le\wdb A$.

Throughout the paper, all complex manifolds are assumed to be connected
(although most of our results can easily be extended to manifolds having finitely
many components of the same dimension). The structure sheaf of a complex manifold $X$
will always be denoted by $\cO_X$. The phrase ``an $\cO_X$-module'' will mean
``a sheaf of $\cO_X$-modules''. Recall that $\cO(X)$ is a nuclear Fr\'echet algebra with
respect to the compact-open topology (i.e., the topology of uniform convergence
on compact subsets of $X$). Recall also (see, e.g., \cite[5.6]{GR_II}) that, for
each coherent $\cO_X$-module $\cF$, the space of global sections $\cF(X)=\Gamma(X,\cF)$
has a canonical topology making $\cF(X)$ into a nuclear Fr\'echet $\cO(X)$-module.

If $X$ is a Stein manifold, then
a Fr\'echet $\cO(X)$-module $F$ is a {\em Stein module} if it is topologically
isomorphic to $\cF(X)$ for some coherent $\cO_X$-module $\cF$.
By a result of O.~Forster~\cite[2.1]{For}, the functor $\Gamma(X,\,\cdot\,)$ of global
sections is an equivalence between the category of coherent $\cO_X$-modules
and the full subcategory of $\cO(X)\lmod$ consisting of Stein $\cO(X)$-modules.

\section{Forster resolution and weak dimension}
\label{sect:weak}

The following lemma is an easy consequence of O.~Forster's results \cite{For}.
We formulate it here for the reader's convenience.

\begin{lemma}
Let $X$ be a Stein manifold, and let $Y\subset X$ be a closed submanifold of
codimension~$m$. There exists an exact sequence
\begin{equation}
\label{For_res_sh}
0\lar \cO_Y \lar \cP_0 \lar \cP_1 \lar \cdots \lar \cP_m \lar 0
\end{equation}
where $\cP_0,\ldots ,\cP_m$ are locally free $\cO_X$-modules.
\end{lemma}
\begin{proof}
By \cite[6.4]{For}, there exists an exact sequence
\[
0\lar \cO_Y \lar \cF_0 \xla{d_0} \cF_1 \lar \cdots \lar \cF_{m-1} \xla{d_{m-1}} \cF_m
\xla{d_m} \cdots
\]
where $\cF_0,\cF_1,\ldots$ are free $\cO_X$-modules.
Letting $\cP_m=\cIm d_{m-1}$, we obtain an exact sequence
\begin{equation}
\label{For_res_sh_1}
0\lar \cO_Y \lar \cF_0 \lar \cF_1 \lar \cdots \lar \cF_{m-1}\lar \cP_m \lar 0.
\end{equation}
To complete the proof, it remains to show that $\cP_m$ is locally free.
Fix $x\in X$, and consider the exact sequence
\begin{equation}
\label{For_res_loc}
0\lar \cO_{Y,x} \lar \cF_{0,x} \lar \cF_{1,x} \lar \cdots \lar
\cF_{m-1,x}\lar \cP_{m,x} \lar 0
\end{equation}
of $\cO_{X,x}$-modules.
If $x\notin Y$, then $\cO_{Y,x}=0$, whence \eqref{For_res_loc} splits,
and so $\cP_{m,x}$ is projective. Now suppose that $x\in Y$, and
let $\cI\subset\cO_X$ denote the ideal sheaf of $Y$.
We have $\cO_{Y,x}\cong\cO_{X,x}/\cI_x$.
Choose a local
coordinate system $z^1,\ldots ,z^n$ in a neighborhood $U$ of $x$ such that
$Y\cap U=\{ z\in U : z^1=\cdots =z^m=0\}$. The $\cO_{X,x}$-module $\cI_x$
is generated by the regular sequence $(z^1,\ldots ,z^m)$, which implies that
$\dh_{\cO_{X,x}}\cO_{Y,x}=m$
(see, e.g., \cite[3.8, Theorem 22]{Northcott2})\footnote[1]{Here, of course,
the projective homological dimension $\dh_{\cO_{X,x}}\cO_{Y,x}$ should be understood
in the purely algebraic context.}.
By looking at \eqref{For_res_loc}, we conclude that $\cP_{m,x}$ is projective.

Thus for each $x\in X$ the $\cO_{X,x}$-module $\cP_{m,x}$ is projective.
Since $\cO_{X,x}$ is local, this means that $\cP_{m,x}$ is free.
Therefore $\cP_m$ is a locally free $\cO_X$-module, as required.
\end{proof}

\begin{corollary}
\label{cor:For-res}
Let $X$ be a Stein manifold, and let $Y\subset X$ be a closed submanifold of
codimension~$m$. There exists an exact sequence
\begin{equation}
\label{For_res}
0\lar \cO(Y) \lar P_0 \lar P_1 \lar \cdots \lar P_m \lar 0
\end{equation}
of Fr\'echet $\cO(X)$-modules,
where $P_0,\ldots ,P_m$ are finitely generated and strictly projective.
\end{corollary}
\begin{proof}
Applying the global section functor $\Gamma(X,\,\cdot\,)$ to \eqref{For_res_sh},
we obtain~\eqref{For_res}. Using \cite[6.2 and 6.3]{For}, we conclude
that $P_0,\ldots ,P_m$ are finitely generated and strictly projective.
\end{proof}

\begin{definition}
Any sequence of the form \eqref{For_res_sh} (respectively, \eqref{For_res})
will be called a {\em Forster resolution} of $\cO_Y$ over $\cO_X$
(respectively, of $\cO(Y)$ over $\cO(X)$).
\end{definition}

\begin{remark}
\label{rem:sh-mod-res}
Since the global section functor is an equivalence between the category
of coherent $\cO_X$-modules and the category of Stein $\cO(X)$-modules,
it yields a 1-1 correspondence between Forster resolutions of
$\cO_Y$ over $\cO_X$ and Forster resolutions of $\cO(Y)$ over $\cO(X)$.
\end{remark}

\begin{remark}
In fact, it easily follows from Forster's construction \cite{For} that resolution
\eqref{For_res} can be chosen in such a way that
$P_0=\cO(X)$ and that the arrow $P_0\to \cO(Y)$ is the restriction map.
We will not use this in the sequel.
\end{remark}

\begin{remark}
Note that \eqref{For_res} is not necessarily
a resolution in the sense of Topological Homology, i.e., it need not be admissible.
The first counterexample was given in \cite[Proposition 5.3]{Mit_Hen};
a more general situation will be discussed in Remark~\ref{rem:nonsplit} below.
\end{remark}

\begin{prop}
\label{prop:Tor-mod-sh}
Let $X$ be a Stein manifold, and let $Y$ be a closed submanifold of $X$.
Then for each coherent $\cO_X$-module $\cF$ and each $p\in\Z_+$ we have
a topological isomorphism
\[
\Tor_p^{\cO(X)}(\cO(Y),\cF(X))\cong\Gamma(X,\cTor_p^{\cO_X}(\cO_Y,\cF)).
\]
\end{prop}
\begin{proof}
Let $\cP\to\cO_Y\to 0$ be a Forster resolution of $\cO_Y$ over $\cO_X$.
Applying the global section functor, we get a Forster resolution
$\cP(X)\to\cO(Y)\to 0$ of $\cO(Y)$ over $\cO(X)$ (see Remark \ref{rem:sh-mod-res}).
Since all the modules $\cP_i(X)$ are nuclear, it follows that
\[
\Tor_p^{\cO(X)}(\cO(Y),\cF(X))\cong H_p(\cP(X)\ptens{\cO(X)} \cF(X)).
\]
We also have $\cP(X)\ptens{\cO(X)} \cF(X)\cong \Gamma(X,\cP\tens{\cO_X}\cF)$
(see, e.g., \cite[4.2.4]{Eschm_Put} or \cite[2.2]{Pir_Stein}). Hence
\[
\begin{split}
\Tor_p^{\cO(X)}(\cO(Y),\cF(X))
&\cong H_p(\Gamma(X,\cP\tens{\cO_X}\cF))\\
&\cong\Gamma(X,H_p(\cP\tens{\cO_X}\cF))
=\Gamma(X,\cTor_p^{\cO_X}(\cO_Y,\cF)).\qedhere
\end{split}
\]
\end{proof}

\begin{corollary}
\label{cor:HKR_Tor}
Let $X$ be a Stein manifold, and let $Y$ be a closed submanifold of $X$.
Denote by $\cI\subset\cO_X$ the ideal sheaf of $Y$. Then for each $p\in\Z_+$
we have
\[
\Tor_p^{\cO(X)}(\cO(Y),\cO(Y))\cong\Gamma(X,{\textstyle\bigwedge}^p (\cI/\cI^2)).
\]
\end{corollary}
\begin{proof}
By \cite[3.5]{Manin_K}, we have $\cTor_p^{\cO_X}(\cO_Y,\cO_Y)
\cong\bigwedge^p (\cI/\cI^2)$ (the proof of this fact, given in
\cite{Manin_K} for regular schemes, applies to complex manifolds without changes).
Now it remains to apply Proposition~\ref{prop:Tor-mod-sh}.
\end{proof}

As a byproduct, we obtain the following analytic version of the
Hochschild--Kostant--Rosenberg Theorem~\cite{HKR}.

\begin{corollary}
\label{cor:HKR}
Let $X$ be a Stein manifold. Then for each $p\in\Z_+$ we have
\[
\cH_p(\cO(X),\cO(X))\cong\Omega^p(X),
\]
where $\Omega^p(X)$ is the space of holomorphic $p$-forms on $X$.
\end{corollary}
\begin{proof}
Recall from \cite[II.3.3]{Groth} that there exists a topological isomorphism
\[
\cO(X)^e=\cO(X)\Ptens\cO(X)\cong\cO(X\times X),\quad
f\otimes g\mapsto ((x,y)\mapsto f(x)g(y)).
\]
Under this identification, the $\cO(X)^e$-module $\cO(X)$ becomes the
$\cO(X\times X)$-module $\cO(\Delta)$, where $\Delta=\{ (x,x) : x\in X\}$ is the
diagonal of $X\times X$. Let $\cI\subset\cO_{X\times X}$
be the ideal sheaf of $\Delta$. Identifying $X$ with $\Delta$ via the map
$x\mapsto (x,x)$ induces a sheaf isomorphism between
$(\cI/\cI^2)|_\Delta$ and the cotangent sheaf $\Omega_X^1$ (see, e.g.,
\cite{Groth_SHC}). Now Corollary~\ref{cor:HKR_Tor} implies that
\[
\begin{split}
\cH_p(\cO(X),\cO(X))=\Tor_p^{\cO(X)^e}(\cO(X),\cO(X))
&\cong\Tor_p^{\cO(X\times X)}(\cO(\Delta),\cO(\Delta))\\
&\cong\Gamma(X\times X,{\textstyle\bigwedge}^p (\cI/\cI^2))
\cong \Omega^p(X).\qedhere
\end{split}
\]
\end{proof}

\begin{theorem}
\label{thm:wdh}
Let $X$ be a Stein manifold, and let $Y$ be a closed submanifold of $X$.
Then $\wdh_{\cO(X)}\cO(Y)=\codim_X Y$.
\end{theorem}
\begin{proof}
Let $M$ be a Fr\'echet $\cO(X)$-module, and let $P\to\cO(Y)\to 0$ be
a Forster resolution of $\cO(Y)$ over $\cO(X)$. Using the nuclearity argument
(see the proof of Proposition~\ref{prop:Tor-mod-sh}), we see that
$\Tor_p^{\cO(X)}(\cO(Y),M)$ is topologically isomorphic to $H_p(P\ptens{\cO(X)}M)$,
which vanishes for $p>m$ and is Hausdorff for $p=m$, where $m=\codim_X Y$.
Therefore $\wdh_{\cO(X)}\cO(Y)\le m$. To obtain the opposite estimate,
recall from \cite[3.4]{Manin_K} that $(\cI/\cI^2)|_Y$ is locally free of rank $m$.
In particular, $\bigwedge^m (\cI/\cI^2)\ne 0$, and it follows from
Corollary~\ref{cor:HKR_Tor} that $\Tor_m^{\cO(X)}(\cO(Y),\cO(Y))\ne 0$.
The rest is clear.
\end{proof}

\begin{corollary}[\cite{Put_sect,Pir_Stein}]
\label{cor:wdg}
Let $X$ be a Stein manifold. Then $\wdg\cO(X)=\wdb\cO(X)=\dim X$.
\end{corollary}
\begin{proof}
As in the proof of Corollary \ref{cor:HKR}, identify $X$ with the
diagonal $\Delta\subset X\times X$.
Theorem~\ref{thm:wdh} implies that
\[
\wdb\cO(X)=\wdh\nolimits_{\cO(X)^e}\cO(X)=\wdh\nolimits_{\cO(X\times X)}\cO(\Delta)
=\codim_{X\times X}\Delta=\dim X.
\]
Hence $\wdg\cO(X)\le\dim X$. Applying Theorem~\ref{thm:wdh} to the
singleton $Y=\{ x_0\}$ yields the opposite estimate $\wdg\cO(X)\ge\codim_X \{ x_0\}=\dim X$.
\end{proof}

\section{Liouville-type property and projective dimension}
\label{sect:Liouv}

Let $X$ be a complex manifold. Following \cite{Palam}, we say that
$X$ is of {\em Liouville type} if each bounded above plurisubharmonic function
on $X$ is constant. For example, each nonsingular affine algebraic
variety is of Liouville type. More examples can be found in
\cite{Aytuna_survey,Aytuna_Sad} (note that Stein manifolds of Liouville type are
called {\em parabolic} in \cite{Aytuna_Sad}).

Starting from this section, we will make use of the linear topological invariants
$(DN)$ and $(\Omega)$ introduced by D.~Vogt~\cite{Vogt_sub,VW_quot}
(see also \cite{MV}). Let $E$ be a Fr\'echet space.
By definition, $E$ has property $(DN)$ if the topology on $E$ can be determined
by an increasing sequence $\{ \|\cdot\|_n : n\in\N\}$ of seminorms satisfying the
following condition: there exists $p\in\N$ such that for each $k\in\N$ and each
$0<r<1$ there exist $n\in\N$ and $C>0$ such that
\[
\| x\|_k\le C\| x\|_p^r \| x\|_n^{1-r}\qquad (x\in E).
\]

Given a continuous seminorm $\|\cdot\|$ on a Fr\'echet space $E$,
define the dual ``seminorm''
$\|\cdot\|^*\colon E^*\to [0,+\infty]$ by $\| y\|^*=\sup\{ |y(x)| : \| x\|\le 1\}$.
Note that $\|\cdot\|^*$ can take the value $+\infty$ as well.
By definition, $E$ has property $(\Omega)$ if the topology on $E$ can be determined
by an increasing sequence $\{ \|\cdot\|_n : n\in\N\}$ of seminorms satisfying the
following condition: for each $p\in\N$ there exists $q\in\N$ such that for each
$k\in\N$ there exist $C>0$ and $r\in (0,1)$ satisfying
\[
\| y\|_q^* \le C (\| y\|_p^*)^r (\| y\|_k^*)^{1-r}\qquad (y\in E^*).
\]

Property $(DN)$ is inherited by subspaces, while $(\Omega)$ is inherited by
quotients.
A basic fact about $(DN)$ and $(\Omega)$ is the following Splitting
Theorem due to Vogt and Wagner \cite{VW_quot} (see also \cite{MV}):
an exact sequence $0\to E\to F\to G\to 0$ of nuclear Fr\'echet spaces
splits provided that $E$ has $(\Omega)$ and
$G$ has $(DN)$.

The following lemma is an easy consequence of the Splitting Theorem.

\begin{lemma}
\label{lemma:split}
Let $E=(E_i,d_i)$ be a chain complex
of nuclear Fr\'echet spaces. Suppose that all the spaces $E_i$
have properties $(DN)$ and $(\Omega)$. Then $E$ splits.
\end{lemma}
\begin{proof}
For each $i\in\Z$ we have an exact sequence
\begin{equation}
\label{short_exact}
0 \lar K_i \xla{d_i} E_{i+1} \lar K_{i+1}\lar 0,
\end{equation}
where $K_i=\Im (d_i\colon E_{i+1}\to E_i)$.
Note that each $K_i$, being a subspace of $E_i$ and a
quotient of $E_{i+1}$, has properties $(DN)$ and $(\Omega)$.
Now the Splitting Theorem implies that \eqref{short_exact} splits for all $i\in\Z$,
which means exactly that $E$ splits.
\end{proof}

Let now $X$ be a Stein manifold. As was observed in \cite{VW2},
$\cO(X)$ always has property $(\Omega)$. Indeed, for $X=\CC^n$ this can be
seen directly, and in the general case $X$ can be embedded into $\CC^n$
for sufficiently large $n$, so $\cO(X)$ becomes a quotient of $\cO(\CC^n)$.
On the other hand, it was shown independently by Zakharyuta~\cite{Zah_iso},
Vogt~\cite{Vogt_some_results}, and Aytuna~\cite{Aytuna_sp_anal} that
$\cO(X)$ has $(DN)$ if and only if $X$ is of Liouville type.
In particular, this is true provided that $X$ is affine algebraic \cite{Zah_alg}.
For more results in this direction, see \cite{Aytuna_survey,Aytuna_Sad,Palam}.

\begin{corollary}
\label{cor:For_adm}
Let $X$ be a Stein manifold, and let $Y$ be a closed submanifold of $X$.
Suppose that both $X$ and $Y$ are of Liouville type.
Then each Forster resolution of $\cO(Y)$ over $\cO(X)$ is admissible.
\end{corollary}
\begin{proof}
Since $X$ and $Y$ are of Liouville type, it follows that
$\cO(X)$ and $\cO(Y)$ have properties $(DN)$ and $(\Omega)$.
Each finitely generated, strictly projective Fr\'echet $\cO(X)$-module
also has properties $(DN)$ and $(\Omega)$, because it is isomorphic to
a direct summand of $\cO(X)^p$ for some $p\in\N$.
Now we see that resolution \eqref{For_res} satisfies the
conditions of Lemma~\ref{lemma:split}.
\end{proof}

\begin{theorem}
\label{thm:dh_Liouv}
Let $X$ be a Stein manifold, and let $Y$ be a closed submanifold of $X$.
Suppose that both $X$ and $Y$ are of Liouville type.
Then $\dh_{\cO(X)}\cO(Y)=\codim_X Y$.
\end{theorem}
\begin{proof}
Corollary~\ref{cor:For_adm} implies that $\dh_{\cO(X)}\cO(Y)\le\codim_X Y$,
and the opposite estimate is immediate from Theorem~\ref{thm:wdh}.
\end{proof}

Arguing in the same way as in the proof of Corollary~\ref{cor:wdg}, we obtain the
following.

\begin{corollary}
\label{cor:dgdb}
Let $X$ be a Stein manifold of Liouville type.
Then $\dg\cO(X)=\db\cO(X)=\dim X$.
\end{corollary}

\section{Van den Bergh isomorphisms for $\cO(X)$}
\label{sect:VdB}

The Van den Bergh isomorphisms are certain relations between Hochschild homology
and cohomology of associative algebras. Special cases of such isomorphisms
can be traced back to the origins of homological algebra, but systematically
they were studied only in 1998 by M.~Van den Bergh \cite{VdB}.
For more recent results involving the Van den Bergh isomorphisms,
see \cite{Far,Ginzb,Ciccoli,DTT,Dolg,Lambre,Menichi}, to cite a few.
Below we will use a Fr\'echet algebra version of
the Van den Bergh isomorphisms, which was introduced
and applied in \cite{Pir_Nova} to some problems of Topological Homology
(see also \cite{Pir_dgdb}).

Let $A$ be a Fr\'echet algebra. A Fr\'echet $A$-bimodule $M$ is {\em invertible}
if there exists a Fr\'echet $A$-bimodule $M^{-1}$ such that
$M\ptens{A} M^{-1}\cong M^{-1}\ptens{A} M\cong A$ in $A\bimod A$.
For example, if $X$ is a Stein manifold, then for each
line sheaf $\cL$ of $\cO_X$-modules the $\cO(X)$-bimodule $\cL(X)$ is
invertible, and $\cL(X)^{-1}=\cL^{-1}(X)$, where
$\cL^{-1}=\cHom_{\cO_X}(\cL,\cO_X)$. More examples can be found in \cite{Pir_dgdb}.
By \cite[Prop. 5.2.6]{Pir_Nova}, each invertible Fr\'echet $A$-bimodule
is projective in $A\lmod$ and in $\rmod A$.

Following \cite{Far} (see also \cite{Pir_dgdb} for the Fr\'echet algebra case),
we say that $A$
{\em satisfies the Van den Bergh condition} $\VdB(n)$ if there exists an invertible
Fr\'echet $A$-bimodule $L$ (a {\em dualizing bimodule})
such that for each Fr\'echet
$A$-bimodule $M$ and each $i\in\Z$ there is a vector space isomorphism
\begin{equation*}
\label{VdB_top}
\cH^i(A,M)\cong\cH_{n-i}(A,L\ptens{A} M).
\end{equation*}
For example, if $X$ is a smooth manifold, then $C^\infty(X)$ satisfies
$\VdB(n)$ with $n=\dim X$ and $L=T^n(X)$, the module of smooth $n$-polyvector
fields on $X$ \cite{Pir_Nova}. We refer to \cite{Pir_dgdb} for more examples.

The goal of this section is to establish the Van den Bergh isomorphisms for
$A=\cO(X)$, where $X$ is a Stein manifold of Liouville type.

\begin{lemma}
\label{lemma:tens_flat}
Let $A$ be a Fr\'echet algebra, and let $C=(C^i,d^i)$ be a cochain complex
in $\rmod A$ such that the cohomology spaces $H^i(C)$ are Hausdorff.
Let $F\in A\lmod$ be a flat Fr\'echet module, and assume that either all the $C^i$'s are nuclear
or $A$ and $F$ are nuclear. Then the cohomology spaces
$H^i(C\ptens{A} F)$ are also Hausdorff, and there exist canonical topological isomorphisms
\begin{equation}
\label{H-tens}
H^i(C)\ptens{A} F\cong H^i(C\ptens{A} F).
\end{equation}
Specifically, if $\xi\in H^i(C)$ is represented by an $i$-cocycle $c\in C^i$,
then \eqref{H-tens} takes $\xi\otimes_A x$ to the cohomology class of $c\otimes_A x$.
\end{lemma}
\begin{proof}
Let $Z^i=Z^i(C)$ denote the space of $i$-cocycles
of $C$. Write also $H^i=H^i(C)$ for short.
For each $i$ we have an exact sequence
\[
0\to Z^i\to C^i\xra{d^i} Z^{i+1}\to H^{i+1}\to 0
\]
of right Fr\'echet $A$-modules. Since $F$ is flat, it follows from the nuclearity
assumptions (see, e.g., \cite[3.1.12]{Eschm_Put}) that the tensored sequence
\begin{equation}
\label{tens_flat_1}
0\to Z^i\ptens{A} F\to C^i\ptens{A} F\xra{d^i\otimes_A\id_F}
Z^{i+1}\ptens{A} F\to H^{i+1}\ptens{A} F\to 0
\end{equation}
is exact. Using the exactness of \eqref{tens_flat_1} first for $i$, and then
for $i+1$, we obtain topological isomorphisms
\begin{equation}
\label{Z_tens}
\begin{split}
Z^i\ptens{A} F
&\cong \Ker\bigl(C^i\ptens{A} F\xra{d^i\otimes_A \id_F} Z^{i+1}\ptens{A} F\bigr)\\
&=\Ker\bigl(C^i\ptens{A} F\xra{d^i\otimes_A \id_F} C^{i+1}\ptens{A} F\bigr)
=Z^i(C\ptens{A} F).
\end{split}
\end{equation}
Now it follows from the exactness of \eqref{tens_flat_1} for $i-1$
and from \eqref{Z_tens} that
\[
\begin{split}
H^i\ptens{A} F
&\cong\Coker\bigl(C^{i-1}\ptens{A} F\xra{d^{i-1}\ptens{A}\id_F} Z^i\ptens{A} F\bigr)\\
&\cong\Coker\bigl(C^{i-1}\ptens{A} F\xra{d^{i-1}\ptens{A}\id_F} Z^i(C\ptens{A} F)\bigr)
=H^i(C\ptens{A} F).
\end{split}
\]
It follows from the construction that, for each $c\in Z^i$ and $x\in F$,
the resulting isomorphism \eqref{H-tens}
indeed takes $(c+\Im d^{i-1})\otimes_A x$ to $c\otimes_A x+\Im(d^{i-1}\otimes_A\id_F)$,
as required.
\end{proof}

In what follows, we will need some topological modules which are not Fr\'echet
modules. Let $A$ be a Fr\'echet algebra.
A {\em left $A$-$\Ptens$-module} is a left $A$-module $M$ endowed with
a complete locally convex topology in such a way that
the action $A\times M\to M$ is continuous. Right $A$-$\Ptens$-modules and
$A$-$\Ptens$-bimodules are defined similarly. If $M$ is a right
$A$-$\Ptens$-module and $N$
is a left $A$-$\Ptens$-module, then their $A$-module tensor product
$M\ptens{A}N$ is defined similarly to the case of Fr\'echet modules
(see Section~\ref{sect:prelim}). The only exception is that
the quotient $(M\Ptens N)/L$ need not be complete in the nonmetrizable case,
so $M\ptens{A}N$ is defined to be the completion of $(M\Ptens N)/L$.

Recall now a construction from \cite{Pir_Nova}. Let $A$ be a Fr\'echet
algebra, and let $M,N\in A\lmod$. Fix a projective resolution $P\to M\to 0$
of $M$ in $A\lmod$. Each $\h_A(P_i,A)$ has a natural structure of a right
$A$-$\Ptens$-module given by $(\varphi\cdot a)(p)=\varphi(p) a$ for
$a\in A$, $p\in P$. We have a map of complexes
\begin{equation}
\label{h-tens}
\h_A(P,A)\ptens{A} N\to\h_A(P,N),\quad
\varphi\otimes_A n\mapsto (p\mapsto\varphi(p)\cdot n).
\end{equation}
If, for some $n$, the spaces $\Ext_A^n(M,A)$ and $\Ext^n_A(M,N)$ are Hausdorff
and complete, then $\Ext^n_A(M,A)$ becomes a right $A$-$\Ptens$-module
in a natural way, and we have continuous linear maps
\begin{equation}
\label{Ext-tens}
\Ext^n_A(M,A)\ptens{A} N\to H^n(\h_A(P,A)\ptens{A} N)\to \Ext^n_A(M,N).
\end{equation}
Here, for each $\xi\in\Ext^n_A(M,A)$ represented by an $n$-cocycle $\varphi\in \h_A(P_n,A)$,
the first arrow in \eqref{Ext-tens} takes $\xi\otimes_A n$ to the cohomology class
of $\varphi\otimes_A n$. The second arrow in \eqref{Ext-tens} is induced by \eqref{h-tens}.
The composite map
\begin{equation}
\label{Ext-Ext}
\Ext^n_A(M,A)\ptens{A} N\to \Ext^n_A(M,N)
\end{equation}
does not depend on the choice of the projective resolution $P$.

\begin{lemma}
\label{lemma:hom-Fre}
Let $A$ be a Fr\'echet algebra, and let $P$ be a finitely generated, strictly
projective left Fr\'echet $A$-module. Then for each $N\in A\lmod$ $\h_A(P,N)$
is a Fr\'echet space. Moreover, if $N$ is nuclear, then so is $\h_A(P,N)$.
\end{lemma}
\begin{proof}
If $P=A$, then $\h_A(A,N)$ is topologically isomorphic to $N$ via the map
$\varphi\mapsto\varphi(1)$; the continuity of the inverse map
$n\mapsto (a\mapsto a\cdot n)$ is easily checked. The general case follows
by additivity and functoriality.
\end{proof}

Following \cite{Pir_dgdb}, we say that a left Fr\'echet $A$-module $M$
is {\em of strictly finite type} if $M$ has a resolution $P\to M\to 0$
consisting of finitely generated, strictly projective Fr\'echet
$A$-modules. The Fr\'echet algebra $A$ is said to be {\em of finite type}
if $A$ is of strictly finite type in $A^e\lmod$.
For example, if $X$ is a Stein manifold, $Y\subset X$ is a closed submanifold,
and both $X$ and $Y$ are of Liouville type, then $\cO(Y)$ is of strictly
finite type over $\cO(X)$ (see Corollary~\ref{cor:For_adm}). Identifying $\cO(X)^e$
with $\cO(X\times X)$, we see, in particular, that the algebra
$\cO(X)$ is of finite type.

\begin{corollary}
\label{cor:Ext-Fre}
Let $A$ be a Fr\'echet algebra, and let $M\in A\lmod$ be of strictly finite type.
If $\Ext_A^n(M,N)$ is Hausdorff for some $n$, then $\Ext_A^n(M,N)$ is a Fr\'echet
space.
\end{corollary}

\begin{lemma}
\label{lemma:Ext-Ext}
Let $A$ be a nuclear Fr\'echet algebra, and let $M\in A\lmod$ be of
strictly finite type. Suppose that $\Ext^n_A(M,A)$ is Hausdorff
for all $n\in\Z_+$. Then for each flat $N\in A\lmod$ and each $n\in\Z_+$
$\Ext^n_A(M,N)$ is Hausdorff as well, and the canonical map \eqref{Ext-Ext}
is a topological isomorphism.
\end{lemma}
\begin{proof}
Fix a resolution $P\to M\to 0$
consisting of finitely generated, strictly projective Fr\'echet
$A$-modules. By Lemma \ref{lemma:hom-Fre}, the cochain complex
$C=\h_A(P,A)$ satisfies the conditions of Lemma \ref{lemma:tens_flat}.
Therefore the first arrow in \eqref{Ext-tens} is a topological isomorphism.
Since each $P_i$ is finitely generated and strictly projective,
it follows that \eqref{h-tens} is also a topological isomorphism
(cf. the proof of Lemma~\ref{lemma:hom-Fre}).
Hence the second arrow in \eqref{Ext-tens} is a topological isomorphism as well.
This completes the proof.
\end{proof}

For the reader's convenience, we recall two results from \cite{Pir_Nova}.

\begin{prop}[{\cite[Prop. 5.2.1]{Pir_Nova}}]
\label{prop:Nova1}
Let $M$ be a left Fr\'echet module of finite projective homological dimension
over a Fr\'echet algebra $A$. Suppose that there exists $n\in\N$
such that for each projective module
$P\in A\lmod$ the following conditions hold:
\begin{mycompactenum}
\item $\Ext^i_A(M,P)=0$ unless $i=n$;
\item $\Ext^n_A(M,P)$ is Hausdorff and complete;
\item the canonical map
\begin{equation}
\label{extmap2}
\Ext^n_A(M,A)\ptens{A}P\to\Ext^n_A(M,P)
\end{equation}
is a topological isomorphism.
\end{mycompactenum}
Put $L=\Ext^n_A(M,A)$.
Then for each $N\in A\lmod$ there exist vector
space isomorphisms
\[
\Ext^i_A(M,N)\cong\Tor_{n-i}^A(L,N).
\]
\end{prop}

\begin{prop}[{\cite[Theorem 5.2.4]{Pir_Nova}}]
\label{prop:Nova2}
Let $A$ be a Fr\'echet algebra of finite bidimension. Suppose that there
exists $n\in\N$ such that for each projective bimodule
$P\in A\bimod A$
the following conditions hold:
\begin{mycompactenum}
\item $\cH^i(A,P)=0$ unless $i=n$;
\item $\cH^n(A,P)$ is Hausdorff and complete;
\item the canonical map
\[
\cH^n(A,A^e)\ptens{A^e}P\to\cH^n(A,P)
\]
is a topological isomorphism.
\end{mycompactenum}
Finally, suppose that $L=\cH^n(A,A^e)$ is projective as a right
$A$-$\Ptens$-module.
Then for each bimodule $M\in A\bimod A$ there exist vector
space isomorphisms
\[
\cH^i(A,M)\cong\cH_{n-i}(A,L\ptens{A}M).
\]
\end{prop}

We can now simplify conditions (i)--(iii) of Propositions \ref{prop:Nova1}
and \ref{prop:Nova2} as follows.

\begin{corollary}
\label{cor:VdB_mod}
Let $M$ be a left Fr\'echet module of strictly finite type and of
finite projective homological dimension
over a nuclear Fr\'echet algebra $A$. Suppose that
there exists $n\in\N$
such that $\Ext^i_A(M,A)=0$ unless $i=n$, and that
$L=\Ext^n_A(M,A)$ is Hausdorff.
Then for each $N\in A\lmod$ there exist vector
space isomorphisms
\[
\Ext^i_A(M,N)\cong\Tor_{n-i}^A(L,N).
\]
\end{corollary}
\begin{proof}
Lemma \ref{lemma:Ext-Ext} shows that if conditions (i) and (ii)
of Proposition~\ref{prop:Nova1} hold for $P=A$, then conditions (i)--(iii)
hold for each flat $P\in A\lmod$. The rest is clear.
\end{proof}

The following result is an analytic version of Van den Bergh's theorem \cite{VdB}.
It follows from Proposition~\ref{prop:Nova2} in exactly the same way as
Corollary~\ref{cor:VdB_mod} follows from Proposition~\ref{prop:Nova1}.

\begin{corollary}
\label{cor:VdB_alg}
Let $A$ be a nuclear Fr\'echet algebra of finite type and of finite bidimension.
Suppose that there
exists $n\in\N$ such that $\cH^i(A,A^e)=0$ unless $i=n$, and that
$L=\cH^n(A,A^e)$ is Hausdorff and invertible as a Fr\'echet $A$-bimodule.
Then $A$ satisfies $\VdB(n)$ with dualizing bimodule $L$.
\end{corollary}

Our next goal is to apply the above results to the algebra $\cO(X)$.

\begin{lemma}
\label{lemma:hom_Gamma}
Let $X$ be a Stein manifold, and let $\cF$ and $\cG$ be coherent $\cO_X$-modules.
If $\cF$ is locally free, then the canonical isomorphism
\[
\h_{\cO(X)}(\cF(X),\cG(X))\cong\Gamma(X,\cHom_{\cO_X}(\cF,\cG))
\]
is a topological isomorphism.
\end{lemma}
\begin{proof}
The case where $\cF\cong\cO_X$ is clear. In the general case
$\cF$ is isomorphic to a direct summand of
$\cO_X^p$ for some $p\in\N$ (see \cite[6.2 and 6.3]{For}), so the result
follows by additivity and functoriality.
\end{proof}

Let $X$ be a complex manifold, and let $Y\subset X$ be a closed submanifold.
Denote by $\cI$ the ideal sheaf of $Y$ in $X$.
Recall that the {\em normal sheaf} of $Y$ in $X$ is defined to be
$\cN_{Y|X}=\cHom_{\cO_Y}(\cI/\cI^2,\cO_Y)$.

\begin{prop}
\label{prop:fund_loc_mod}
Let $X$ be a Stein manifold, and let $Y\subset X$ be a closed submanifold of
codimension~$m$.
Suppose that both $X$ and $Y$ are of Liouville type.
Then for each flat Fr\'echet $\cO(X)$-module $F$ and each $i\in\Z_+$
there exist topological isomorphisms
\begin{equation}
\label{fund_loc_mod}
\Ext^i_{\cO(X)}(\cO(Y),F)\cong
\begin{cases}
\Gamma(X,{\textstyle\bigwedge^m}\cN_{Y|X})\ptens{\cO(X)} F & \text{ for } i=m,\\
0 & \text{ otherwise.}
\end{cases}
\end{equation}
\end{prop}
\begin{proof}
Let $\cP\to\cO_Y\to 0$ be a Forster resolution of $\cO_Y$ over $\cO_X$.
By \cite[III.7.2]{RD}, we have
\[
\cExt^i_{\cO_X}(\cO_Y,\cO_X)\cong
\begin{cases}
{\textstyle\bigwedge^m}\cN_{Y|X} & \text{ for } i=m,\\
0 & \text{ otherwise}
\end{cases}
\]
(the proof of this fact, given in
\cite{RD} for schemes, applies to complex manifolds without changes).
Therefore we have an exact sequence
\[
\cHom_{\cO_X}(\cP,\cO_X)\to {\textstyle\bigwedge^m}\cN_{Y|X}\to 0.
\]
Applying the section functor and taking into account Lemma~\ref{lemma:hom_Gamma},
we obtain an exact sequence
\[
\h_{\cO(X)}(\cP(X),\cO(X))\to\Gamma(X,{\textstyle\bigwedge^m}\cN_{Y|X})\to 0
\]
of Fr\'echet $\cO(X)$-modules.
This yields isomorphisms \eqref{fund_loc_mod} for $F=\cO(X)$.
The general case follows from Lemma~\ref{lemma:Ext-Ext}.
\end{proof}

Let $X$ be a complex manifold embedded into $X\times X$ via the diagonal embedding.
Recall that the normal sheaf $\cN_{X|X\times X}$ is then isomorphic to the
tangent sheaf $\cT_X$ of $X$
(cf. the proof of Corollary~\ref{cor:HKR}). For each $n\in\N$, let
$T^n(X)=\Gamma(X,\bigwedge^n\cT_X)$ denote the space of holomorphic
$n$-polyvector fields on $X$. Now Proposition~\ref{prop:fund_loc_mod}
implies the following.

\begin{prop}
\label{prop:fund_loc_alg}
Let $X$ be an $n$-dimensional Stein manifold of Liouville type.
Then for each flat Fr\'echet $\cO(X)$-bimodule $F$ and each $i\in\Z_+$
there exist topological isomorphisms
\begin{equation*}
\cH^i(\cO(X),F)\cong
\begin{cases}
T^n(X)\ptens{\cO(X)^e} F & \text{ for } i=n,\\
0 & \text{ otherwise.}
\end{cases}
\end{equation*}
\end{prop}

\begin{theorem}
Let $X$ be a Stein manifold, and let $Y\subset X$ be a closed submanifold of
codimension~$m$.
Suppose that both $X$ and $Y$ are of Liouville type.
Then for each Fr\'echet $\cO(X)$-module $M$ there exist vector space
isomorphisms
\begin{equation}
\label{Ext-Tor-O(X)}
\Ext^i_{\cO(X)}(\cO(Y),M)\cong
\Tor_{m-i}^{\cO(X)}\bigl(\Gamma(X,{\textstyle\bigwedge^m}\cN_{Y|X}),M\bigr).
\end{equation}
\end{theorem}
\begin{proof}
Apply Proposition~\ref{prop:fund_loc_mod} and Corollary~\ref{cor:VdB_mod}.
\end{proof}

\begin{theorem}
\label{thm:VdB}
Let $X$ be an $n$-dimensional Stein manifold of Liouville type.
Then $\cO(X)$ satisfies $\VdB(n)$ with dualizing bimodule $T^n(X)$.
\end{theorem}
\begin{proof}
Apply Proposition~\ref{prop:fund_loc_alg} and Corollary~\ref{cor:VdB_alg}.
\end{proof}

\begin{remark}
\label{rem:wdh_C^infty}
In \cite[5.4.1]{Pir_Nova}, we proved \eqref{Ext-Tor-O(X)} in the situation
where $X$ is a smooth real manifold, $\cO(X)=C^\infty(X)$ is the Fr\'echet
algebra of smooth functions on $X$, and $Y\subset X$ is a closed submanifold
of codimension $m$.
This clearly implies that $\dh_{C^\infty(X)} C^\infty(Y)=m$
\cite[5.4.2]{Pir_Nova}. The fact that $\wdh_{C^\infty(X)} C^\infty(Y)=m$
was not stated explicitly in \cite{Pir_Nova}, but it also easily follows from
the above-mentioned ``smooth'' version of~\eqref{Ext-Tor-O(X)}.
Indeed, letting $A=C^\infty(X)$, $B=C^\infty(Y)$, and
$L=\Gamma(X,{\textstyle\bigwedge^m}\cN_{Y|X})$, we obtain
\[
\Tor_m^A(L,B)\cong\Ext_A^0(B,B)=\h_A(B,B)=\h_B(B,B)\cong B\ne 0,
\]
whence $\wdh_A B\ge m$. Since $\dh_A B=m$ (see above), we conclude that $\wdh_A B=m$
as well.
\end{remark}

\section{Hyperconvex submanifolds and projective dimension}
\label{sect:hyper}

Let $X$ be a complex manifold. Recall from \cite{Stehle} that
$X$ is {\em hyperconvex} if there exists a plurisubharmonic function
$\rho\colon X\to [-\infty,0)$ such that for each $c<0$ the set
$\ol{\{ x\in X : \rho(x)\le c\}}$ is compact.
For example, each analytic polyhedron in a Stein manifold \cite{Stehle}
and each relatively compact pseudoconvex domain with
a $C^2$ boundary in a Stein manifold \cite{Died_For} are hyperconvex.
More examples can be found in \cite{Aytuna_survey}.
Note that a hyperconvex manifold $X$ can never be of Liouville type
(unless $\dim X=0$).

By definition \cite{Vogt_pow_fin} (see also \cite{MV}),
a Fr\'echet space $E$ has property $(\ol{\Omega})$
if the topology on $E$ can be determined
by an increasing sequence $\{ \|\cdot\|_n : n\in\N\}$ of seminorms satisfying the
following condition: for each $p\in\N$ there exists $q\in\N$ such that for each
$k\in\N$ there exists $C>0$ satisfying
\[
\| y\|_q^* \le C (\| y\|_p^*)^{1/2} (\| y\|_k^*)^{1/2}\qquad (y\in E^*).
\]
Clearly, $(\ol{\Omega})$ implies $(\Omega)$ (see Section~\ref{sect:Liouv}),
but not vice versa. If $X$ is a Stein manifold, then $\cO(X)$ has
$(\ol{\Omega})$ if and only if $X$ is hyperconvex \cite{Zah_iso,Aytuna_sp_anal}.
For more results in this area, we refer to \cite{Aytuna_survey}.

Let $E$ and $F$ be locally convex spaces. Following \cite{Vogt_bdd}, we say that
a linear map $\varphi\colon E\to F$ is {\em bounded} if there exists a
$0$-neighborhood $U\subset E$ such that $\varphi(U)\subset F$ is bounded.
Clearly, each bounded linear map is continuous.
On the other hand, the identity map on a locally convex space $E$ is bounded
if and only if $E$ is normable. In this section, we will use the following important
result by D.~Vogt~\cite{Vogt_bdd} (see also \cite[29.21]{MV}):
if $E$ is a Fr\'echet space with $(\ol{\Omega})$
and $F$ is a Fr\'echet space with $(DN)$, then each continuous linear map
from $E$ to $F$ is bounded.

The set of all bounded
linear maps from $E$ to $F$ is denoted by $\cLB(E,F)$; obviously, this is a vector
subspace of $\cL(E,F)$. Note also that $\cLB(E,F)$, like $\cL(E,F)$, is a bifunctor
on the category of locally convex spaces with values in the category of vector spaces.

Recall that an inverse system $(E_\lambda,\tau^\lambda_\mu)$ of locally convex spaces
is {\em reduced} if the canonical projections
$\tau_\mu\colon \varprojlim E_\lambda\to E_\mu$ have dense ranges.
A {\em reduced inverse limit} is the inverse limit of a reduced inverse system
of locally convex spaces.

\begin{lemma}
\label{lemma:LB_lim}
Let $E=\varprojlim E_\lambda$ be a reduced inverse limit of locally convex spaces.
Then for each Hausdorff l.c.s. $F$ the canonical maps $\cLB(E_\lambda,F)\to\cLB(E,F)$
yield a vector space isomorphism
\begin{equation}
\label{LB_lim}
\varinjlim\cLB(E_\lambda,F)\to\cLB(E,F).
\end{equation}
\end{lemma}
\begin{proof}
Since the projections $\tau_\lambda\colon E\to E_\lambda$ have dense ranges,
it follows that the maps $\cLB(E_\lambda,F)\to\cLB(E,F)$ are injective,
and hence \eqref{LB_lim} is injective as well.
Let now $\varphi\in\cLB(E,F)$, and let $U\subset E$ be a $0$-neighborhood such that
$B=\varphi(U)$ is bounded in $F$. Without loss of generality we may assume that
$U=\tau_\lambda^{-1}(V)$ for some $\lambda$ and for some open $0$-neighborhood
$V\subset E_\lambda$. Set $E_\lambda^0=\Im\tau_\lambda\subset E_\lambda$,
and define $\varphi_\lambda^0\colon E_\lambda^0\to F$ by
\[
\varphi_\lambda^0(\tau_\lambda(x))=\varphi(x)\qquad (x\in E).
\]
To see that $\varphi_\lambda^0$ is well defined, assume that $\tau_\lambda(x)=0$.
Then for each $\eps>0$ we have $\tau_\lambda(x)\in\eps V$, whence $x\in\eps U$
and $\varphi(x)\in\eps B$. Since $F$ is Hausdorff, this implies that $\varphi(x)=0$,
showing that $\varphi_\lambda^0$ is well defined. Now observe that
\[
\varphi_\lambda^0(V\cap E_\lambda^0)=\varphi_\lambda^0(\tau_\lambda(U))=B,
\]
and so $\varphi_\lambda^0$ is bounded. Since $E_\lambda^0$ is dense in $E_\lambda$,
$\varphi_\lambda^0$ uniquely extends to a continuous linear map
$\varphi_\lambda\colon E_\lambda\to F$. Moreover, since $V\cap E_\lambda^0$
is dense in $V$, it follows that
$\varphi_\lambda(V)\subset\overline{B}$, and so $\varphi_\lambda$ is bounded.
It remains to observe that $\varphi$ is the image of $\varphi_\lambda$ under
the canonical map $\cLB(E_\lambda,F)\to\cLB(E,F)$. Therefore \eqref{LB_lim} is onto.
\end{proof}

Before formulating the next result, recall from \cite[II.3.1]{Groth} that,
if $E$ is a Banach space and $F$ is a nuclear Fr\'echet space, then
there exists a topological isomorphism
\begin{equation}
\label{Groth_iso}
F\Ptens E^*\to\cL(E,F),\quad
y\otimes f\mapsto (x\mapsto f(x)y).
\end{equation}
If, in addition, $E$ and $F$ are left Fr\'echet $A$-modules, then it is easy to check
that \eqref{Groth_iso} is an $A$-bimodule isomorphism.

\begin{lemma}
\label{lemma:Ext_hi}
Let $A$ be a Fr\'echet-Arens-Michael algebra of finite type.
Suppose that $A$ satisfies $\VdB(n)$ with nuclear dualizing bimodule $L$.
Let now $M\in A\lmod$, and assume that all continuous linear maps from
$M$ to $L^{-1}$ are bounded. Then there exists a vector space isomorphism
\begin{equation}
\label{Ext_hi}
\Ext_A^n(M,L^{-1})\cong M^*.
\end{equation}
As a consequence, if $M\ne 0$, then $\dh_A M=n$.
\end{lemma}
\begin{proof}
First suppose that $M$ is a Banach $A$-module.
Using \eqref{Tor_Hoch}, \eqref{Ext_Hoch}, property $\VdB(n)$,
and \eqref{Groth_iso}, we obtain the
following vector space isomorphisms:
\[
\begin{split}
\Ext^n_A(M,L^{-1})
&\cong \cH^n(A,\cL(M,L^{-1}))
\cong\cH_0(A,L\ptens{A}\cL(M,L^{-1}))\\
&\cong\cH_0(A,L\ptens{A} L^{-1}\Ptens M^*)
\cong\cH_0(A,A\Ptens M^*)
\cong\Tor_0^A(M^*,A)\cong M^*.
\end{split}
\]
This proves \eqref{Ext_hi} in the case where $M$ is a Banach module.

In the general case, represent $M$ as a reduced inverse limit
$M=\varprojlim M_\lambda$ of left Banach $A$-modules (see \cite[Prop. 3.5]{Pir_qfree}).
By Lemma~\ref{lemma:LB_lim}, we have vector space isomorphisms
\[
\cL(M,L^{-1})=\cLB(M,L^{-1})\cong\varinjlim\cLB(M_\lambda,L^{-1})
=\varinjlim\cL(M_\lambda,L^{-1}).
\]
In other words, for $P=A\Ptens A$ we have a vector space isomorphism
\begin{equation}
\label{h_inj}
\h_A(P\ptens{A} M,L^{-1})\cong \varinjlim\h_A(P\ptens{A} M_\lambda, L^{-1}).
\end{equation}
By additivity and functoriality, \eqref{h_inj} holds for each finitely generated,
strictly projective Fr\'echet $A$-bimodule $P$.

Now let $P\to A\to 0$ be a resolution of $A$ in $A\bimod A$ such that all the $P_i$'s
are finitely generated and strictly projective.
Then for each $N\in A\lmod$ $P\ptens{A} N\to N\to 0$ is a projective resolution of
$N$ in $A\lmod$. Using \eqref{h_inj} and the
already established isomorphism \eqref{Ext_hi} for the Banach modules $M_\lambda$,
we obtain vector space isomorphisms
\[
\begin{split}
\Ext^n_A(M,L^{-1})
=H^n(\h_A(P\ptens{A} M,L^{-1}))
&\cong\varinjlim H^n(\h_A(P\ptens{A} M_\lambda,L^{-1}))\\
&=\varinjlim\Ext^n_A(M_\lambda,L^{-1})
\cong\varinjlim M_\lambda^*
\cong M^*.
\end{split}
\]
This proves \eqref{Ext_hi}, which in turn implies that $\dh_A M\ge n$ provided that $M\ne 0$.
On the other hand, $\VdB(n)$ readily implies that\footnote[1]{In fact, it is easy
to show that if $A$ satisfies $\VdB(n)$, then
$\db A=\wdb A=n$ \cite{Pir_dgdb}. If, in addition, $A$ is nuclear and
locally $m$-convex, then $\dg A=\wdg A=n$ as well [loc. cit.].} $\db A\le n$,
whence $\dh_A M\le n$. This completes the proof.
\end{proof}

\begin{theorem}
\label{thm:dh_hyper_Liouv}
Let $X$ be a Stein manifold, and let $Y$ be a closed submanifold of $X$.
Suppose that $X$ is of Liouville type and that $Y$ is hyperconvex.
Then $\dh_{\cO(X)}\cO(Y)=\dim X$.
\end{theorem}
\begin{proof}
Since $X$ is of Liouville type, it follows that $\cO(X)$ has
property $(DN)$ \cite{Zah_iso,Vogt_some_results,Aytuna_sp_anal}.
Moreover, if $\cF$ is a locally free $\cO_X$-module, then $\cF$ is a direct
summand of $\cO_X^p$ for some $p$ \cite[6.2 and 6.3]{For},
and so $\cF(X)$ has property $(DN)$. In particular, the space
$\Omega^n(X)$ of holomorphic $n$-forms has property $(DN)$, where $n=\dim X$.
On the other hand, since $Y$ is hyperconvex, it follows that $\cO(Y)$ has property
$(\ol{\Omega})$ \cite{Zah_iso,Aytuna_sp_anal}, and so
all continuous linear maps from $\cO(Y)$ to $\Omega^n(X)$ are bounded \cite{Vogt_bdd}.
Taking into account Theorem~\ref{thm:VdB}, we see that the algebra $A=\cO(X)$
and the modules $M=\cO(Y)$, $L^{-1}=\Omega^n(X)$ satisfy the conditions
of Lemma~\ref{lemma:Ext_hi}. Therefore $\dh_{\cO(X)}\cO(Y)=n$, as required.
\end{proof}

\begin{remark}
\label{rem:nonsplit}
By comparing Corollary~\ref{cor:For-res} and Theorem~\ref{thm:dh_hyper_Liouv},
we see that, under the conditions of Theorem~\ref{thm:dh_hyper_Liouv},
there is no {\em admissible} Forster resolution
of $\cO(Y)$ over $\cO(X)$ (unless $\dim Y=0$). This can also be seen directly
as follows. If  resolution \eqref{For_res} were admissible,
then the underlying Fr\'echet space of $\cO(Y)$
would be a direct summand of $P_0$ and would have $(\ol{\Omega})$ and
$(DN)$ simultaneously, which is impossible for a nonnormable space.
\end{remark}

\begin{example}
Let $Y$ be a nonsingular affine algebraic variety in $\CC^n$.
Using Theorems~\ref{thm:wdh} and~\ref{thm:dh_Liouv}, we see that
\[
\dh\nolimits_{\cO(\CC^n)}\cO(Y)=\wdh\nolimits_{\cO(\CC^n)}\cO(Y)=n-\dim Y.
\]
Let now $Y$ be any hyperconvex Stein manifold (for example, a polydisk or a
ball in $\CC^m$). By the Remmert-Bishop-Narasimhan Embedding Theorem, $Y$ can be embedded
into $\CC^n$ for $n=2\dim Y+1$. Using Theorems~\ref{thm:wdh}
and~\ref{thm:dh_hyper_Liouv}, we see that
\[
\wdh\nolimits_{\cO(\CC^n)}\cO(Y)=n-\dim Y,\quad\text{but}\;
\dh\nolimits_{\cO(\CC^n)}\cO(Y)=n.
\]
\end{example}

In conclusion, let us formulate an open problem due to A.~Ya.~Helemskii \cite{X_31}:

\begin{problem}[cf. Corollaries~\ref{cor:wdg} and~\ref{cor:dgdb}]
Let $X$ be a Stein manifold. Is it true that
$\dg\cO(X)=\dim X$ or $\db\cO(X)=\dim X$?
\end{problem}

\begin{ackn}
The author thanks A.~Ya.~Helemskii and P.~Doma\'nski for helpful discussions.
\end{ackn}

\end{document}